\documentclass[11pt,reqno]{imsart}
\usepackage{amsmath,amssymb,amsthm}
\usepackage{comment}
\usepackage{xcolor}
\usepackage{fullpage}
\usepackage[hiresbb]{graphicx}
\usepackage{graphicx}
\allowdisplaybreaks
\usepackage[round]{natbib}
\usepackage{latexsym}
\usepackage{lmodern}
\usepackage[T1]{fontenc}
\DeclareFontShape{T1}{lmr}{m}{scit}{<->ssub * lmr/m/it}{}

\newtheorem{lemma}{Lemma}

\newtheorem{theorem}{Theorem}
\newtheorem{corollary}{Corollary}

\newtheorem{remark}{Remark}

\newcommand{\mH}[2]{
\raise1.5pt\hbox{$\displaystyle \mathop{H}^{\rm m}$}\hspace{-1.5pt}_{#1}^{\, #2}}
\newcommand{\eH}[2]{
\raise1.5pt\hbox{$\displaystyle \mathop{H}^{\rm e}$}
\hspace{-1.5pt}_{#1}^{\, #2}}

\newcommand{\mGamma}[2]{
\raise1.5pt\hbox{$\displaystyle \mathop{\Gamma}^{\rm m}$}\hspace{-1.5pt}_{#1}^{\, #2}}
\newcommand{\eGamma}[2]{
\raise1.5pt\hbox{$\displaystyle \mathop{\Gamma}^{\rm e}$}
\hspace{-1.5pt}_{#1}^{\, #2}}

\begin{document}
\begin{frontmatter}
\begin{aug}
\title{Duality induced by an embedding structure of determinantal point process}
\author{
\fnms{Hideitsu} \snm{Hino}\ead[label=e1]{hino@ism.ac.jp}}
\and
\author{\fnms{Keisuke} \snm{Yano}\ead[label=e2]{yano@ism.ac.jp}}
 \address{The Institute of Statistical Mathematics, \\ 10-3 Midori cho, Tachikawa City, Tokyo, 190-8562, Japan. \\
\printead{e1,e2}}
\end{aug}
\begin{abstract}
This paper investigates the information geometrical structure of a determinantal point process (DPP). It demonstrates that a DPP is embedded in the exponential family of log-linear models. The extent of deviation from an exponential family is analyzed using the $\mathrm{e}$-embedding curvature tensor, which identifies partially flat parameters of a DPP.
On the basis of this embedding structure, the duality related to a marginal kernel and an $L$-ensemble kernel is discovered.
\end{abstract}
\end{frontmatter}

\section{Introduction}
\label{sec:introduction}

Determinantal point processes (DPP) on a finite set are a common choice for modeling probabilities that accommodate repulsive property \citep{KuleszaandTasker_2012}.
They have found wide application such as document summarization \citep{Lin_Bilmes_2012_UAI}, recommendation systems \citep{Gartrell_etal_2016_RecSys}, and Bayesian statistics \citep{Kojima_Komaki_2016_SS,Bardenet_Hardy_2020_AAP,RockovaandGeorge_2022}, being blessed with efficient algorithms for inferential tasks \citep{Gillenwater_etal_2014,Affandi_etal_2014,Mariet_Sra_2016_NIPS,Urschel_etal_2017,dupuy2018,Kawashima_Hino_2023_TMLR}.

In this paper, we elucidate the information geometric structure of a DPP model. Specifically, we clarify the embedding structure of a DPP model
in the exponential family of log-linear models (c.f.,~\citealp{Agresti1990,Amari2001}) in Theorem \ref{thm:DPPiscurvedexp}.
Models embedded in exponential families are called curved exponential families.
Information geometry \citep{amari1985differential} provides a measure, the $\mathrm{e}$-embedding curvature tensor \citep{Efron_1975_AoS,Reeds_1975,Amari_1982_AoS,Sei_2011}, to quantify the extent to which a curved exponential family deviates from an exponential family. 
To check the $\mathrm{e}$-embedding curvature as well as the Fisher information matrix,
we apply the diagonal scaling \citep{Marshall_Olkin_1968_NM}, also known as the quality vs.~diversity decomposition in the DPP literature \citep{KuleszaandTasker_2012}, to an $L$-ensemble kernel of a DPP model and then
evaluate them, which clarifies that the subset of parameters related to the item-wise effects (quality terms)
has zero $\mathrm{e}$-embedding curvature (Corollary \ref{cor: e-curvature}).
This partially $\mathrm{e}$-embedding flat structure endows the Fisher information matrix with a nice representation using the Hessian of a certain function (Corollary \ref{cor: Fisher w.r.t. u}).
This also presents a duality related to two representations of a DPP (Lemma \ref{lemma: expectation parameter}), a marginal kernel representation and $L$-ensemble kernel representation.

As a related work,
\cite{Brunel_etal_2017_colt} investigate the qualitative structure of log-likelihood of a DPP, which enables us to check the identifiability and evaluate the null space of the Fisher information matrix. \cite{Grigorescu_etal_2022} investigate the maximization of the log-likelihood from the computational viewpoint, and \cite{Friedman_etal_2024} further analyze the number of critical points of the log-likelihood using algebraic geometry.
Our analyses complement \cite{Brunel_etal_2017_colt} 
from yet another perspective, the embedding structure of a DPP.
Also, \cite{Batmanghelich2014-fm} utilize the idea that DPP has an exponential form
to variational inference. As \cite{Batmanghelich2014-fm} mention, DPP does not form an exponential family in general.
Our results measure the extent to which a DPP deviates from an exponential family and identify the sub structure that can be regarded as an exponential family.

The rest of the paper is organized in the followings.
Section \ref{sec: preliminaries} summarizes DPPs and information geometry.
Section \ref{sec: IG} presents the main results, and Section \ref{sec: examples} examines the main results through examples.
Section \ref{sec: proofs} gives the proofs of the results.

\section{Preliminaries}
\label{sec: preliminaries}

Before presenting the main results,
we summarize preliminaries used in this paper.

\subsection{Determinantal point process}

Let $\mathcal{Y}:=\{1,\ldots,m\}$ be a finite set of size $m$.
A determinantal point process (DPP) on $\mathcal{Y}$ is a random subset $Y$ of $\mathcal{Y}$ of which distribution is given by
\begin{align*}
P(A\subseteq Y)=\mathrm{det}(K_{A}) \quad ( A\subseteq \mathcal{Y})
\end{align*}
for some $m\times m$ symmetric real matrix $K$ with all eigenvalues in $[0,1]$,
where $K_{B}$ for a subset $B\subseteq \mathcal{Y}$ denotes the submatrix of $K$ with indices of columns and rows lying in $B$.
This matrix $K$ is called a marginal kernel.
The condition that all eigenvalues of $K$ lie in $[0,1]$ is necessary to ensure that a DPP forms a probability distribution (see Section 2 of \cite{KuleszaandTasker_2012} as well as Section 2 of \cite{Kawashima_Hino_2023_TMLR}).
If all eigenvalues of $K$ lie in $(0,1)$, the probability mass function of a DPP with a marginal kernel $K$ has a useful representation \citep{borodin2005}
\begin{align*}
P_L(Y=A)=\cfrac{\det(L_{A})}{\det(L+I_{m \times m})}
\quad (A\subseteq \mathcal{Y})
\quad\text{with}\,\, L:=K(I_{m \times m}-K)^{-1}.
\end{align*}
Here, $I_{m\times m}$ denotes the identity matrix of size $m$. DPP having a marginal kernel with all eigenvalues in $(0,1)$ allows this representation and is called $L$-ensemble.
The matrix $L$ is called an $L$-ensemble kernel.
We consider $L$-ensembles
and consider the parametric model
\[
\mathcal{P}_{\mathrm{DPP}}:=
\left\{
P_{L}(Y=A)=\frac{\mathrm{det}(L_{A})}{\mathrm{det}(L+I_{m\times m})}\,:\,\text{$L$ is a positive definite matrix}
\right\}.
\]

\subsection{Exponential family and curved exponential family}

To introduce our theoretical results,
we briefly review exponential families and curved exponential families (on $\mathcal{Y}$).

An exponential family on $(\mathcal{Y}\,,\,2^{\mathcal{Y}})$ is a family $\{P_{\theta}(\cdot) \,:\,\theta = (\theta^{1},\ldots,\theta^{i},\ldots,\theta^{d})^{\top}\in\Theta \subseteq \mathbb{R}^{d}\}$ of probability functions parameterized by $d$-dimensional parameter $\theta$ that are given by
\[
P_{\theta}(A)=\exp\left(
\sum_{i=1}^{d}\theta^{i}T_{i}(A)
-\phi(\theta)
\right) \quad (A \subseteq \mathcal{Y}),
\]
where $T_{i}(A)\, (i=1,\dots, d)$ are a set of measurable functions called sufficient statistics
and $\phi(\theta)$ is a potential function:
\[
\phi(\theta):=\log \sum_{A\subseteq\mathcal{Y}}
\exp\left(
\sum_{i=1}^{d}\theta^{i}T_{i}(A)
\right).
\]
The parameterization $\theta$ is called a canonical parameter.
In an exponential family, the Fisher information matrix $(g_{ij}(\theta))_{1\le i,j\le d}$ with respect to $\theta$ is given by
\begin{align}
g_{ij}(\theta) = 
\frac{\partial^{2}}{\partial \theta^{i}\partial \theta^{j}}\phi(\theta)
=\mathrm{Cov}_{\theta}[T_{i}(A),T_{j}(A)],
\label{eq: Fisher information of exp}
\end{align}
where $\mathrm{Cov}_{\theta}[\cdot\,,\,\cdot]$ is the covariance with respect to $P_{\theta}$.
From the Legendre duality by the convexity of $\phi(\theta)$, 
we can define the dual parameterization
$\eta=(\eta_{i})_{i=1,\ldots,d}$ of $\theta$ called the expectation parameter by
$\eta_{i}=\partial_{i}\phi(\theta)$.
This parameterization is actually the expectation of the sufficient statistics with respect to $P_{\theta}$:
\begin{align}
\eta_{i}=\mathrm{E}_{\theta}[T_{i}(A)],
\label{eq: expectation parameter}
\end{align}
where $\mathrm{E}_{\theta}[\cdot]$ is the expectation with respect to $P_{\theta}$.
The Fisher information matrix $(g^{ij}(\eta))_{1\le i,j\le d}$ with respect to $\eta$ is given by
\begin{align*}
g^{ij}(\eta)=(g^{-1}(\theta(\eta)))_{ij},
\end{align*}
that is, the $(i,j)$-component of the inverse of the Fisher information matrix $g(\theta)$ with respect to $\theta$.

These two parameterizations $\theta$ and $\eta$ play important roles in geometric structure of statistical models~\citep{amari1985differential}.
For a parametric model $\{P_{\xi}(\cdot):\xi=(\xi^{1},\ldots,\xi^{d})^{\top}\in\Xi\subset\mathbb{R}^{d}\}$,
a parameterization $\xi$ is called $\mathrm{e}$-affine if it vanishes $\mathrm{e}$-connection coefficients $\eGamma{}{}(\xi)$ defined by
\begin{align*}
\eGamma{ijk}{}(\xi) = \sum_{A\subseteq \mathcal{Y}}
P_{\xi}(A)
\left[ 
\frac{\partial^{2}}{ \partial \xi^{i} \partial \xi^{j}}
\log P_{\xi}(A)
\frac{\partial}{\partial \xi^{k}}
\log P_{\xi}(A) \right].
\end{align*}
In the same model, a parameterization $\nu=\nu(\xi)$ is called $\mathrm{m}$-affine if it vanishes $\mathrm{m}$-connection coefficients $\mGamma{}{}(\nu)$ defined by
\begin{align*}
\mGamma{ijk}{}(\nu) = \sum_{A\subseteq\mathcal{Y}} \frac{1}{P_{\xi(\nu)}(A)}
\left[
\frac{\partial^{2}}{\partial \nu^{i}\partial \nu^{j}}
 P_{\xi(\nu)}(A)
 \frac{\partial}{\partial \nu^{k}}
 P_{\xi(\nu)}(A)
 \right].
\end{align*}
In the exponential family, the canonical parameter $\theta$ is $\mathrm{e}$-affine and 
the expectation parameter $\eta$ is $\mathrm{m}$-affine.

A curved exponential family on $(\mathcal{Y},2^{\mathcal{Y}})$ is a family
$\{P_{u}(\cdot)\,:\,u=(u^{1},\ldots,u^{a},\ldots,u^{d'})^{\top}\in\mathcal{U}\subseteq\mathbb{R}^{d'}\}$ of probability functions parameterized by $d'$-dimensional parameter that is smoothly embedded in an exponential family. Here, the embedding of a model is defined through the embedding of the parameter. Consider a curved exponential family $\{P_{u}(\cdot)\,:\,u\in\mathcal{U}\}$ that is embedded in a $d$-dimensional exponential family $\{P_{\theta}(\cdot)\,:\,\theta\in\Theta\}$ with a canonical parameterization $\theta=(\theta^{1},\cdots,\theta^{i},\cdots,\theta^{d})^{\top}$ and an expectation parameterization $\eta=(\eta_{1},\cdots,\eta_{i},\cdots,\eta_{d})^{\top}$. Then, $\theta$ and $\eta$ are smooth functions of $u$: $\theta=\theta(u)$ and $\eta=\eta(u)$. In this case, 
the Jacobian matrices of $\theta$ and $\eta$ with respect to $u$
\begin{align*}
B^{i}_{a}(u)=\frac{\partial \theta^{i}}{\partial u^{a}}\Big{|}_{\theta=\theta(u)}
\quad\text{and}\quad
B_{ia}(u)=\frac{\partial \eta_{i}}{\partial u^{a}}\Big{|}_{\eta=\eta(u)}
\end{align*}
are important tools to describe geometric quantities.
The Fisher information matrix $(g_{ab}(u))_{1\le a,b\le d'}$ 
with respect to $u$ is expressed as
\begin{align}
    g_{ab}(u)=\sum_{i=1}^{d}\sum_{j=1}^{d} B^{i}_{a}B^{j}_{b}g_{ij}(\theta(u))=
    \sum_{i=1}^{d}\sum_{j=1}^{d} B_{ia}B_{jb}g^{ij}(\theta(u))
    \label{eq: Fisher w.r.t.u}
\end{align}
and the $\mathrm{e}/\mathrm{m}$-connection coefficients $\eGamma{}{}(u)$ and $\mGamma{}{}(u)$ are given by
\begin{align*}
\eGamma{abc}{}(u)=
\sum_{i=1}^{d}\sum_{j=1}^{d}
\frac{\partial B^{i}_{b}}{\partial u^{a}}
B^{j}_{c}g_{ij}(\theta(u))
\quad\text{and}\quad
\mGamma{abc}{}(u)=
\sum_{i=1}^{d}\sum_{j=1}^{d}
\frac{\partial B_{ib}}{\partial u^{a}}
B_{jc}g^{ij}(\eta(u)).
\end{align*}
The embedding structure also yield the $\mathrm{e}$-embedding curvature tensor ($\mathrm{e}$-Euler--Schouten embedding curvature).
To define this, we introduce an  ancillary parameter $v=(v^{d'+1},\ldots,v^{\kappa},\ldots, v^{d})$ to each point $\theta(u)$ so that the pair $(u,v)$ forms a parameterization of $\{P_{\theta}(\cdot):\theta\in\Theta\}$ with $(u,0)$ implying $\theta(u)$,
and 
the $(u,v)$-components of 
the Fisher information matrix $g(u,v)$ with respect to $(u,v)$ are zero. 
Using an ancillary parameter, we define the $\mathrm{e}$-embedding curvature $\eH{ab\kappa}{}$ ($1\le a,b\le d'$ and $d'+1\le \kappa \le d$) as
\begin{align}
\eH{ab\kappa}{}(u)=
\sum_{A\subseteq \mathcal{Y}}
P_{\xi}(A)
\left[ 
\frac{\partial^{2}}{ \partial u^{a} \partial u^{b}}
\log P_{(u,0)}(A)
\frac{\partial}{\partial v^{\kappa}}
\log P_{(u,0)}(A) \right]
=
\sum_{i=1}^{d}
\frac{\partial B^{i}_{b}}{\partial u^{a}}
\frac{\partial \eta_{i}}{\partial v^{\kappa}}\Big{|}_{v=0}.
\label{eq: e-curvature}
\end{align}
The $d'\times d'$-matrix 
\[
[\eH{}{}]^{2}_{ab}(u):=
\sum_{\kappa  = d'+1}^{d}
\sum_{\lambda = d'+1}^{d}
\sum_{c = 1}^{d'}
\sum_{d = 1}^{d'}
\eH{ac\kappa}{}(u)\eH{bd\lambda}{}(u)g^{\kappa\lambda}(u,0)g^{cd}(u,0)
\]
represents the square of $\mathrm{e}$-embedding curvature \citep{Efron_1975_AoS},
where $g^{\kappa\lambda}$ and $g^{cd}$ are the inverses of $g_{\kappa\lambda}$ and $g_{cd}$, respectively.
The $\mathrm{e}$-embedding curvature quantifies the extent to which the model deviates from an exponential family.
In fact, the following holds.
\begin{lemma}[\cite{Amari_Nagaoka} and \cite{Sei_2011}]
    Under suitable conditions, 
    the square of $\mathrm{e}$-embedding curvature vanishes if and only if the model is an exponential family.
\end{lemma}

Although
the $\mathrm{m}$-embedding curvature is defined in the same way, we do not discuss it in this paper and omit it.

\section{Information-geometric structure of determinantal point processes}
\label{sec: IG}

This section elucidates the embedding structure of DPP. This connects DPP to the log-linear model (c.f.,~\citealp{Agresti1990,Amari2001}) and endows the DPP model with interesting properties (Corollaries \ref{cor: e-curvature} and \ref{cor: Fisher w.r.t. u}) and the duality behind a marginal kernel $K$ and an $L$-ensemble kernel $L$ (Lemma \ref{lemma: expectation parameter}).

\subsection{Embedding structure of determinantal point processes}
To present the results, we start with employing the diagonal scaling of a positive definite kernel $L$ \citep{Marshall_Olkin_1968_NM}, also known as quality vs.~diversity decomposition \citep{KuleszaandTasker_2012}.
Since $L$ is positive definite,
we parameterize $L$ as $L=DRD$ with the diagonal matrix $D$ and the symmetric matrix $R$, where
\begin{align*}
D:=
\begin{pmatrix}
\sqrt{L_{11}} & 0 & \cdots & 0\\
0 & \sqrt{L_{22}}  & \cdots & 0\\
\vdots    & \vdots    & \vdots & \vdots \\
0 & 0 & \cdots & 0\\
0 & 0 & \cdots & \sqrt{L_{mm}} \\
\end{pmatrix}
\,\text{and}\,
R:=
\begin{pmatrix}
1         & \rho_{12} & \cdots & \rho_{1m}\\
\rho_{12} & 1         & \cdots & \rho_{2m}\\
\vdots    & \vdots    & \vdots & \vdots \\
\rho_{1(m-1)} & \rho_{2(m-1)}         & \cdots & \rho_{(m-1)m}\\
\rho_{1m}         & \rho_{2m}         & \cdots & 1 \\
\end{pmatrix}
\end{align*}
with $\rho_{ij}:=L_{ij}/\sqrt{L_{ii}L_{jj}}$.

The following theorem gives us an embedding structure of DPP as a curved exponential family. 
The exponential family in which DPP is embedded is a log-linear model.

\begin{theorem}
\label{thm:DPPiscurvedexp}
(a) The DPP model $\mathcal{P}_{\mathrm{DPP}}$ is written as the model 
$\mathcal{P}:=\{P_{u}(A)\,:\,u\in\mathbb{R}^{m}\times (-\infty,0)^{m(m-1)/2}\}$ parameterized by an $m(m+1)/2$-dimensional parameter 
\[
u=\left(u^{\{1\}},\ldots,u^{\{m\}},u^{\{1,2\}},
u^{\{1,3\}},
\ldots,u^{\{(m-1),m\}}\right)
\]
with the indices sorted in a lexicographical order
of which member has the form 
\begin{align}
\label{eq:cExpDPP}
P_{u}(A)
&=\exp\left(\sum_{\alpha_{1}\in\mathcal{S}_{1}^{m}} u^{\alpha_{1}}T_{\alpha_{1}}(A)+
\sum_{ \alpha_{2}\in\mathcal{S}_{2}^{m}}u^{\alpha_{2}}T_{\alpha_{2}}(A)
+\sum_{k=3}^{m}\sum_{\mathcal{I}_{k}\in\mathcal{S}_{k}^{m}}
\theta^{\mathcal{I}_{k}}(u)T_{\mathcal{I}_{k}}(A)
-\psi(u)\right)
\end{align}
where 
\begin{align*}
\mathcal{S}_{k}^{m}
&:=
\left\{\{i_{1},\ldots,i_{k}\}\,:\,
1\le i_{1}<\ldots i_{k}\le m\right\}
&(1\le k \le m),
\\
T_{\mathcal{I}_{k}}(A)&:=1_{i_{1}\in A\,,\,\ldots,i_{k}\in  A}
\quad\text{with}\quad
\mathcal{I}_{k}=\{i_{1},\ldots,i_{k}\}
&(\mathcal{I}_{k}\in\mathcal{S}_{k}^{m}\,\,\text{and}\,\,1\le k \le m),\\
u^{\alpha_{1}}&:=\log  L_{\alpha_{1}} &(\alpha_{1}\in\mathcal{S}_{1}^{m}),\\
u^{\alpha_{2}}&:=\log (1-\rho_{\alpha_{2}}^{2}) &(\alpha_{2}\in\mathcal{S}_{2}^{m}),\\
\psi(u)&:=\sum_{\alpha_{1}\in\mathcal{S}_{1}^{m}}u^{\alpha_{1}}+
\log\mathrm{det}\begin{pmatrix}
e^{-u^{\{1\}}}+1 & \cdots & \sqrt{1-e^{u^{\{1,m\}}}} \\
\vdots       & \vdots & \vdots \\
\sqrt{1-e^{u^{\{1,m\}}}} & \cdots & e^{-u^{\{m\}}}+1
\end{pmatrix}
,
\end{align*}
and
$\theta^{\mathcal{I}_{k}}(u)$ ($\mathcal{I}_{k}=(i_{1},\ldots,i_{k})\in\mathcal{S}_{k}^{m}$
 and $3\le k \le m$)
is defined recursively as
\begin{equation}
\begin{split}
\theta^{\mathcal{I}_{k}}(u)&:
=
\log \mathrm{det}
\left(I_{k\times k}+
\left(
1_{a \ne b}
\sqrt{1-e^{u^{\{a,b\}}}}\right)_{a,b=i_{1},\ldots,i_{k}}\right)
\\
&\quad\quad -
\sum_{l=3}^{k-1}\sum_{
\mathcal{J}_{l}\in\mathcal{S}_{l}^{m}
\,:\,
\mathcal{J}_{l}\subset \mathcal{I}_{k}
}\theta^{\mathcal{J}_{l}}(u)
-\sum_{
\alpha_{2}\in\mathcal{S}_{2}^{m}
\,:\,
\alpha_{2}\subset \mathcal{I}_{k}}u^{\alpha_{2}}.
\end{split}
\label{eq:recTheta}
\end{equation}

(b) The family $\mathcal{P}$ is the curved exponential family embedded in the exponential family $\mathcal{E}:=\{P_{\theta}(A)\,:\,\theta\in\mathbb{R}^{2^m-1}\}$ given by
\begin{align}
\label{eq:ExpDPP}
P_{\theta}(A)=
\exp\left(
\sum_{k=1}^{m}
\sum_{\mathcal{I}_{k}\in\mathcal{S}_{k}^{m}}
\theta^{\mathcal{I}_{k}}T_{\mathcal{I}_{k}}(A)-\phi(\theta)
\right),
\end{align}
where $\phi(\theta)$ is given by
\begin{align*}
\phi(\theta):= \log \sum_{A \subseteq \mathcal{Y}} \exp \left(
\sum_{k=1}^{m}
\sum_{\mathcal{I}_{k}\in\mathcal{S}_{k}^{m}}
\theta^{\mathcal{I}_{k}}T_{\mathcal{I}_{k}}(A) \right).
\end{align*}
\end{theorem}

Note that 
the DPP model forms an exponential family for $m=1,2$ but do not for $m\ge 3$.
The underlying exponential family of the DPP model is nothing but the log-linear model (c.f.,~\citealp{Agresti1990,Amari2001}); more specifically, the log-linear model on a partially ordered set \citep{Sugiyama_Nahakara_Tsuda_2016}.
The proof of this theorem is given in Section \ref{sec: proofs}.

\begin{figure}[htbp]
    \centering
\includegraphics[scale=0.45]{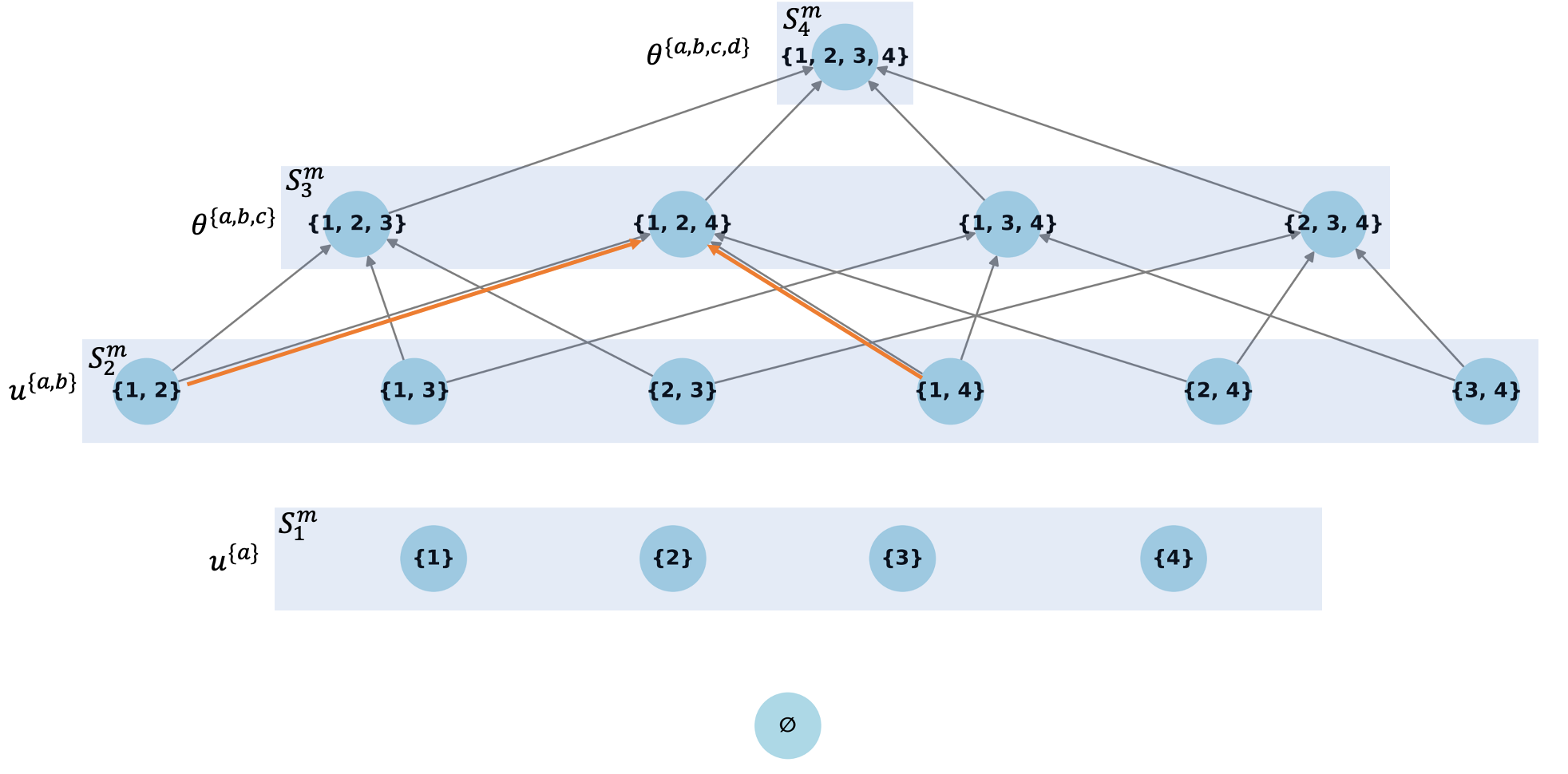}
    \caption{Construction of parameters $(u^{\alpha})_{\alpha\in\mathcal{S}_{1}^{m}\cup\mathcal{S}_{2}^{m}}$ and $(\theta^{\mathcal{I}})_{\mathcal{I}\in\mathcal{S}_{3}^{m}\cup\cdots\cup \mathcal{S}_{m}^{m}}$ in the DPP model with $m=4$. After drawing the Hasse diagram for the power set,
    we eliminate upward paths starting from singletons and the emptyset. In the resulting diagram, nodes correspond to the indices for parameters, and directed edges represent the dependence of parameters. 
    Nodes in the same layer have the same cardinality. 
    For example, $\theta^{\{1,2,4\}}$ is constructed using parameters having directed edges to $\{1,2,4\}$ as indicated in orange edges.
    }
    \label{fig: curvedexponential}
\end{figure}

Figure \ref{fig: curvedexponential} displays the construction of $u$ and $\theta(u)$ using a modification of the Hasse diagram, where the figure is drawn for $m=4$.
We begin with drawing the Hasse diagram of the power set $2^{\mathcal{Y}}$ of an $m$-element item $\mathcal{Y}=\{1,\ldots,m\}$ ordered by inclusion. 
A node of the Hasse diagram corresponds to a subset of the power set, and a directed edge from a subset $A$ to a subset $B$ is drawn when
$A\ne B$, $A\subset B$, and there is no $C$ distinct from $A$ and $B$ with $A\subset C\subset B$. A directed edge in the Hasse diagram is called an upward path.
We then eliminate all upward paths starting from singletons $\{m\}$ ($1\le a \le m$) and the empty set $\emptyset$.
In this modified Hasse diagram, the second and the third layers from the bottom correspond to the indices for the free parameter $u$. The layers from the third to the top correspond to the indices for $\theta(u)$.
The recursive construction (\ref{eq:recTheta}) in Theorem \ref{thm:DPPiscurvedexp} is explained by upward paths in the diagram. Each element $\theta^{\mathcal{I}_{k}}(u)$ 
($\mathcal{I}_{k}\in\mathcal{S}^{m}_{k}$, $3\le k \le m$)
consists of 
\[\log \mathrm{det}(R_{\mathcal{I}_{k}})
=\log \mathrm{det}
\left(I_{k\times k}+
\left(
1_{a \ne b}
\sqrt{1-e^{u^{\{a,b\}}}}\right)_{a,b=i_{1},\ldots,i_{k}}\right)
\]
subtracted by 
all $u^{\alpha_{2}}$
with $\alpha_{2}$ 
in the second layer 
that are connected to $\mathcal{I}_{k}$
via upward paths
and 
all $\theta^{\mathcal{J}_{l}}$
with $\mathcal{J}_{l}$
in the $l$-th layer ($3\le l\le k-1$)
that are connected to $\mathcal{I}_{k}$
via upward paths.
For example, 
in the DPP model with $m=4$, the constrained parameteter
$\theta^{\{1,2,4\}}$ is constructed using 
parameters having upward paths to $\{1,2,3\}$ depicted in the orange edges in Figure \ref{fig: curvedexponential}.
Note that parameters $(u^{\alpha_{1}})_{\alpha_{1}\in\mathcal{S}_{1}^{m}}$ in the second layer, corresponding to the quality terms in the DPP literature, are not involved in the construction, which is a key to the following analysis.

\begin{remark}[Unidentifiability]
    The parameterization using $u$ implies the essential unidentifiability issue of DPPs posed by \cite{Brunel_etal_2017_colt}. In fact, $u^{\{a,b\}}$ only recovers $\pm\rho_{ab}$ ($1\le a<b\le m$), which implies that two DPPs $P_{L}$ and $P_{L'}$ are identical if and only if there exists an $m\times m$ diagonal matrix $\tilde{D}$ with $\pm 1$ diagonal entries such that $L'=\tilde{D}L\tilde{D}$.
\end{remark}

We next investigate a structure of the $\mathrm{e}$-embedding curvature.
Although the $\mathrm{e}$-embedding curvature of the DPP model is not necessarily identically zero,
it along with the direction of $\partial/\partial u^{\alpha_{1}}$ ($\alpha_{1}\in\mathcal{S}_{1}^{m}$) is shown to vanish,
which implies that 
$(u^{\alpha_{1}})_{\alpha_{1}\in\mathcal{S}_{1}^{m}}$ can be regarded as the parameter of an exponential family.
\begin{corollary}
\label{cor: e-curvature}
    The $\mathrm{e}$-embedding curvature
    $\eH{\alpha_{1}\alpha\kappa}{}$ 
    vanishes for any $\alpha_{1}\in\mathcal{S}_{1}^{m}$,
    $\alpha \in \mathcal{S}_{1}^{m} \cup \mathcal{S}_{2}^{m}$,
    $\kappa\in\mathcal{S}_{k}^{m}$ ($k=3,\ldots,m$),
    and thus the square of the $\mathrm{e}$-embedding curvature has the following structure:
    \begin{align*}
    \left([\eH{}{}]^{2}_{\alpha\tilde{\alpha}}\right)_{\alpha,\tilde{\alpha}
    \in \mathcal{S}_{1}^{m}\cup \mathcal{S}_{2}^{m}
    }=
    \begin{pmatrix}
    0_{m \times m} & 0_{m\times m(m-1)/2} \\
    0_{m(m-1)/2 \times m} & *
    \end{pmatrix},
    \end{align*}
    where 
    $0_{a\times b}$ ($a,b\in\mathbb{N}$) is the $a \times b$ zero matrix, and
    $*$ implies an $m(m-1)/2\times m(m-1)/2$-matrix that is possibly non-zero.
    \end{corollary}
The proof is given in Section \ref{sec: proofs}.

Blessed with the embedding structure indicated by Corollary \ref{cor: e-curvature},
we obtain the following expression for the Fisher information matrix with respect to $u$.
\begin{corollary}
\label{cor: Fisher w.r.t. u}
    (a) For $1\le a,b\le m$,
    the $(\{a\},\{b\})$-component of the Fisher information matrix $\mathcal{G}$ with respect to $u$ is given by
    \begin{align*}    \mathcal{G}_{\{a\}\{b\}}(u)=
    \frac{\partial^{2}}{\partial u^{\{a\}}\partial u^{\{b\}} }\psi(u)=
    \begin{cases}
    K_{aa}(1-K_{aa})
    &
    (a=b),
    \\
    -K_{ab}^{2}
    &
    (a\ne b).
    \end{cases}
    \end{align*}
    (b) For $\alpha_{1}\in\mathcal{S}_{1}^{m}$ and $\alpha_{2}\in\mathcal{S}_{2}^{m}$,
        the $(\alpha_{1},\alpha_{2})$-component of the Fisher information matrix $\mathcal{G}$ with respect to $u$ is given by
    \begin{align*}
    \mathcal{G}_{\alpha_{1}\alpha_{2}}(u)=
    \mathrm{det}(K_{\alpha_{1}\cup\alpha_{2}})
    -
    \mathrm{det}(K_{\alpha_{1}})
    \mathrm{det}(K_{\alpha_{2}})
    +\sum_{\mathcal{B}_{n}\,:\,
    \alpha_{2} \subsetneq \mathcal{B}_{n}\,,\, 2< n \le m}
    \{\mathrm{det}(K_{\alpha_{2}})
    \left(1-
    \mathrm{det}(K_{\mathcal{B}_{n}})
    \right)
    \}\frac{\partial \theta^{\mathcal{B}_{n}}}{\partial u^{\alpha_{2}}}.
    \end{align*}
\end{corollary}

A key tool for this result is
the covariance expression (\ref{eq: Fisher information of exp}) of the Fisher information matrix for the underlying exponential family of the DPP model evaluated at a DPP: Let $\mathcal{I},
\mathcal{J}
\in \cup_{k=1}^{m}\mathcal{S}_{k}^{m}$.
At the distribution in the log-linear model (\ref{eq:ExpDPP}) corresponding to a DPP 
(that is, $\theta=\theta(u)$),
the $(\mathcal{I},\mathcal{J})$-component of the Fisher information matrix $\mathcal{G}(\theta)$ with respect to $\theta$ of the log-linear model (\ref{eq:ExpDPP}) is given 
\begin{align*}
\mathcal{G}_{\mathcal{I} \mathcal{J}}(\theta)\Big{|}_{\theta=\theta(u)}
=\mathrm{Cov}_{\theta(u)}[T_{\mathcal{I}}
,T_{\mathcal{J}}
]
=\mathrm{det}(K_{\mathcal{I}\cup\mathcal{J}})-\mathrm{det}(K_{\mathcal{I}})
\mathrm{det}(K_{\mathcal{J}}).
\end{align*}
The proof of this result is given in Section \ref{sec: proofs}.

\begin{remark}[Conditional potential function]    
The first part of Corollary \ref{cor: Fisher w.r.t. u} conveys an implication on the role of $\psi(u)$: that is, the Fisher information matrix for $(u^{\alpha_{1}} )_{\alpha_{1}\in\mathcal{S}_{1}^{m}}$ is characterized by the sub-matrix of Hessian of $\psi(u)$.
Because of this characterization, we call $\psi(u)$ a conditional potential function.
This property is not obvious since the DPP model is not an exponential family.
In fact, this characterization using the Hessian matrix is not applicable for the whole Fisher information matrix:
$\mathcal{G}(u)\ne \nabla^{2}_{u}\psi(u)$ in general.
Blessed with the $\mathrm{e}$-embedding flat structure of $(u^{\alpha_{1}})_{\alpha_{1}\in\mathcal{S}_{1}^{m}}$, the part of the Fisher information matrix related to $(u^{\alpha_{1}})_{\alpha_{1}\in\mathcal{S}_{1}^{m}}$ has this property.
\end{remark}

\begin{remark}[Non-orthogonality]
From Corollary \ref{cor: Fisher w.r.t. u} (b), item-wise effects (quality terms)
$(u^{\alpha_{1}})_{\alpha_{1}\in\mathcal{S}_{1}^{m}}$
and pair-wise effects (diversity terms) $(u^{\alpha_{2}})_{\alpha_{2}\in\mathcal{S}_{2}^{m}}$ are not necessarily orthogonal in the Fisher information matrix:
for $\alpha_{1}\in\mathcal{S}_{1}^{m}$ and $\alpha_{2}\in\mathcal{S}_{2}^{m}$,
$\mathcal{G}_{\alpha_{1}\alpha_{2}}(u)\ne 0$ in general.
A disappointing message from this is that estimation of $(u^{\alpha_{1}})_{\alpha_{1}\in\mathcal{S}_{1}^{m}}$ and that of $(u^{\alpha_{2}})_{\alpha_{2}\in\mathcal{S}_{2}^{m}}$ cannot be separated to attain efficient estimation.
\end{remark}

\subsection{Duality in determinantal point processes}

Theorem \ref{thm:DPPiscurvedexp} clarifies the role of an $L$-ensemble kernel in describing the DPP model as a curved exponential family.
What is a role of a marginal kernel $K$ in describing a DPP as a statistical model?
The following lemma answers this question and
implies that
the expectation parameter $\eta=\eta(u)$ of the underlying exponential family (\ref{eq:ExpDPP}) restricted to the DPP model is given by a marginal kernel $K$.
This relationship between the $K$ and $L$
 matrices exhibits a duality within the DPP model, revealing how both matrices play pivotal roles in the model's geometric structure.

\begin{lemma}
\label{lemma: expectation parameter}
    The expectation parameter of (\ref{eq:ExpDPP})
    evaluated at a DPP
    is written as 
    \[
    \eta_{\mathcal{I}_{k}}
    =
    \mathrm{det}(K_{\mathcal{I}_{k}}) \quad (\mathcal{I}_{k}\in \mathcal{S}_{k}^{m},\quad k=1,\ldots,m).
    \]
    Further, the expectation parameter $\eta_{\alpha_{1}}$ for $\alpha_{1}\in\mathcal{S}_{1}^{m}$ evaluated at a DPP is written as 
    \begin{align}
    \eta_{\alpha_{1}} =(\partial/ \partial u^{\alpha_{1}}) \psi(u).
    \label{eq: eta is grad phi}
    \end{align}
\end{lemma}
\begin{proof}
Since $\eta_{\mathcal{I}_{k}}$ is equal to $\mathrm{E}_{\theta}[T_{\mathcal{I}_{k}}]$, and $\mathrm{E}_{\theta}[T_{\mathcal{I}_{k}}]$ at a DPP is equal to $\mathrm{det}(K_{\mathcal{I}_{k}})$,
we conclude the first assertion.
The score vector of the DPP model with respect to $(u^{\alpha_{1}})_{\alpha_{1}\in\mathcal{S}_{1}^{m}}$
is given by
\[
(\partial / \partial u^{\alpha_{1}}) \log P_{u}(A)
= T_{\alpha_{1}}(A)- (\partial/ \partial u^{\alpha_{1}}) \psi(u)
\quad (\alpha_{1}\in\mathcal{S}_{1}^{m}).
\]
Then taking the expectation of both sides yields
\[
0= \eta_{\alpha_{1}}-(\partial/ \partial u^{\alpha_{1}}) \psi(u),
\quad (\alpha_{1}\in\mathcal{S}_{1}^{m})
\]
which concludes the second assertion.
\end{proof}

\begin{remark}[Relation to the Laplace expansion]
    The latter claim is also verified via  the Laplace expansion.
    We take $\eta_{\{1\}}=K_{11}$ as an example.
    By working with the diagonal scaling $L=DRD$, $K_{11}$ has the form
    \[
    K_{11}=\left(
    L(L+I_{m\times m})^{-1}
    \right)_{11}
    =\left(DR(R+D^{-2})^{-1}D^{-1}\right)_{11}
    =\left(R(R+D^{-2})^{-1}\right)_{11}.
    \]
    Let $\Delta_{ij}$ ($1\le i,j\le m$) denote the $(i,j)$-cofactor of $(R+D^{-2})$.
    Using the cofactors of $(R+D^{-2})$, we get the following expression of $K_{11}$:
    \[
K_{11}=
\sum_{j=1}^{m}R_{1j}\left((R+D^{-2})^{-1}\right)_{j1}
=
\sum_{j=1}^{m}R_{1j}\frac{\Delta_{j1}}{\mathrm{det}(R+D^{-2})}
=
\frac{\Delta_{11}+\sum_{j=2}^{m}\rho_{1j}\Delta_{j1}}{\mathrm{det}(R+D^{-2})}.
    \]
    Note that the Laplace expansion yields
    \[
    \mathrm{det}(R+D^{-2})=
    \sum_{j=1}^{m}(R+D^{-2})_{1j}\Delta_{j1}
    =
    (1+e^{-u^{\{1\}}})\Delta_{11}+\sum_{j=2}^{m}\rho_{1j}\Delta_{j1},
    \]
    which gives 
    \begin{align}
K_{11}
=\frac{\left\{(1+e^{-u^{\{1\}}})
\Delta_{11}+\sum_{j=2}^{m}\rho_{1j}\Delta_{j1}\right\}
-e^{ -u^{\{1\}}}\Delta_{11}
}{\mathrm{det}(R+D^{-2})}
    =1-e^{-u^{\{1\}}}\frac{\Delta_{11}}{\mathrm{det}(R+D^{-2})}.
    \label{eq: expression of K}
    \end{align}
    At the same time, the partial derivative $(\partial/\partial u^{\{1\}}) \psi(u)$ is written as
    \begin{align}
    (\partial/\partial u^{\{1\}}) \psi(u)
    &=
    (\partial/\partial u^{\{1\}})\{
    \log \mathrm{det}D^{2} + \log \mathrm{det}(R+D^{-2})\}\nonumber\\
    &=1+\frac{\partial (R+D^{-2})_{11}}{ \partial u^{\{1\}}}\frac{\Delta_{11}}{\mathrm{det}(R+D^{-2})}
    \nonumber\\
    &=1-e^{ -u^{\{2\}}}\frac{\Delta_{11}}{\mathrm{det}(R+D^{-2})}.
    \label{eq: derivative of psi}
    \end{align}
    Thus, 
    combining (\ref{eq: expression of K}) with (\ref{eq: derivative of psi}) also verifies the second claim of Lemma \ref{lemma: expectation parameter}.
\end{remark}

Duality revealed by Lemma \ref{lemma: expectation parameter} suggests us yet another parameterization of the DPP model:
\[
\omega:=(\eta_{\{1\}},\ldots,\eta_{\{m\}}\,;\,u^{\{1,2\}},\ldots,u^{\{(m-1),m\}}).
\]
This type of parameterization has been often used in the analyses of the log-linear model 
\citep{Amari2001,Sugiyama_Nahakara_Tsuda_2016}
and more generally models with the dual foliation structure such as the minimum information dependence model \citep{Sei_Yano_2024}.
It is called a mixed parameterization.
It may not be obvious that the mixed parameterization is well-defined
since the DPP model is a curved exponential family.
By employing a property of the conditional potential function $\psi(u)$,
we ensure the well-definedness of the mixed parameterization $\omega$ in the DPP model as follows.
\begin{lemma}
    There is a one-to-one correspondence between $\omega$ and $u$.
\end{lemma}
\begin{proof}
Given $(u^{\alpha_{2}})_{\alpha_{2}\in\mathcal{S}_{2}^{m}}$,
$\psi(u)$ is strictly convex by Corollary \ref{cor: Fisher w.r.t. u}.
This, together with Lemma \ref{lemma: expectation parameter} implies that
the function $(u^{\alpha_{1}})_{\alpha_{1}\in\mathcal{S}_{1}^{m}}\mapsto ( \partial \phi(u) /\partial u^{\alpha_{1}} )_{\alpha_{1}\in\mathcal{S}_{1}^{m} } $ is injective,
which completes the proof.
\end{proof}

\begin{remark}[The Kullback--Leiber divergence]
Using the mixed parameterization, we obtain an interesting formula for the Kullback--Leibler divergence between DPPs. Let $u$ and $\tilde{u}$ be two points in $\mathbb{R}^{m}\times (-\infty,0)^{m(m-1)/2}$ with $u^{\alpha_{2}}=\tilde{u}^{\alpha_{2}}$ ($\alpha_{2}\in\mathcal{S}_{2}^{m}$). 
Then we have 
\begin{align*}
D[\,P_{u} \,:\, P_{\tilde{u}}\,]
:=\sum_{A\subseteq \mathcal{Y}}P_{u}(A)\log \frac{P_{u}(A)}{P_{\tilde{u}} (A)}
=\psi^{*}(\omega)+\psi(\tilde{u})-
\sum_{\alpha_{1}\in\mathcal{S}_{1}^{m}}\tilde{u}^{\alpha_{1}}\eta_{\alpha_{1}},
\end{align*}
where $\omega$ denotes the mixed parameterization of $u$, and $\psi^{*}(\omega)$ is the Legendre transform of $\psi(u)$ given $(u^{\alpha_{2}})_{\alpha_{2}\in\mathcal{S}_{2}^{m}}$:
\[
\psi^{*}(\omega)=\max_{ u^{\alpha_{1}} \in \mathbb{R}^{m}} \left\{ 
u^{\alpha_{1}}\eta_{\alpha_{1}} - \psi(u)
\right\}.
\]
This type of expression appears in 
the Kullback--Leibler divergence for an exponential family. 
Endowed with the $\mathrm{e}$-embedding flat structure (Corollary \ref{cor: e-curvature}) of $(u^{\alpha_{1}})_{\alpha_{1}\in\mathcal{S}_{1}^{{m}}}$, 
the Kullback--Leibler divergence for the DPP model partially enjoys this expression.
\end{remark}

\section{Examples}
\label{sec: examples}
This section analytically checks the results by using DPP with $m=2$.
It also analyzes the structure of DPP with $m=3$ by using the numerical computation.

\subsection{Determinantal point process with $m=2$}

We begin with DPP with $m=2$.
Observe that 
with the diagonal scaling $L=DRD$,
all probabilities are given by
\begin{align*}
P_{L}(\emptyset)&=e^{-c(L)},
\quad
P_{L}(\{1\})=
e^{\log L_{11}-c(L)},
\quad
P_{L}(\{2\})=
e^{\log L_{22}-c(L)},\\
P_{L}(\{1,2\})&=e^{\log L_{11}+\log L_{22}
+\log|1-\rho_{12}^{2}|-c(L)},
\end{align*}
where $c(L):=\log \mathrm{det} (L+I_{2\times 2})$.
This implies that the DPP model with $m=2$ forms an exponential family
\begin{align*}
P_{L}(A)=P_{u}(A)=e^{u^{\{1\}}T_{\{1\}}(A)
+u^{\{2\}}T_{\{2\}}(A)
+u^{\{1,2\}}T_{\{1,2\}}(A)
-\psi(u)
},
\end{align*}
where
\begin{align*}
T_{\{1\}}(A)&=1_{1\in A},\quad
T_{\{2\}}(A)=1_{2\in A},\quad
T_{\{1,2\}}(A)=1_{1,2\in A}\quad \text{and}\\
u^{\{1\}}&=\log L_{11},\quad
u^{\{2\}}=\log L_{22},\quad
u^{\{1,2\}}=\log (1-\rho_{12}^{2})\\
\psi(u)&=u^{\{1\}}+u^{\{2\}}+\log 
\left\{e^{u^{\{1,2\}}} + e^{-u^{\{1\}}} + e^{-u^{\{2\}}}+ e^{-u^{\{1\}}-u^{\{2\}}} \right\}.
\end{align*}
The conditional potential function $\psi(u)$ becomes an actual potential function.
Figure \ref{fig: conditional potential} (a) displays the contour plot of $\psi(u)$ with respect to $u^{\{1\}}$ and $u^{\{2\}}$ at $u^{\{1,2\}}=-0.1$, from which we see the convex shape of $\psi(u)$ with respect to $u^{\{1\}}$ and $u^{\{2\}}$.
The $\mathrm{e}$-embedding curvature vanishes because this DPP model is an exponential family.

\begin{figure}[htbp]
    \centering
\includegraphics[scale=0.5]{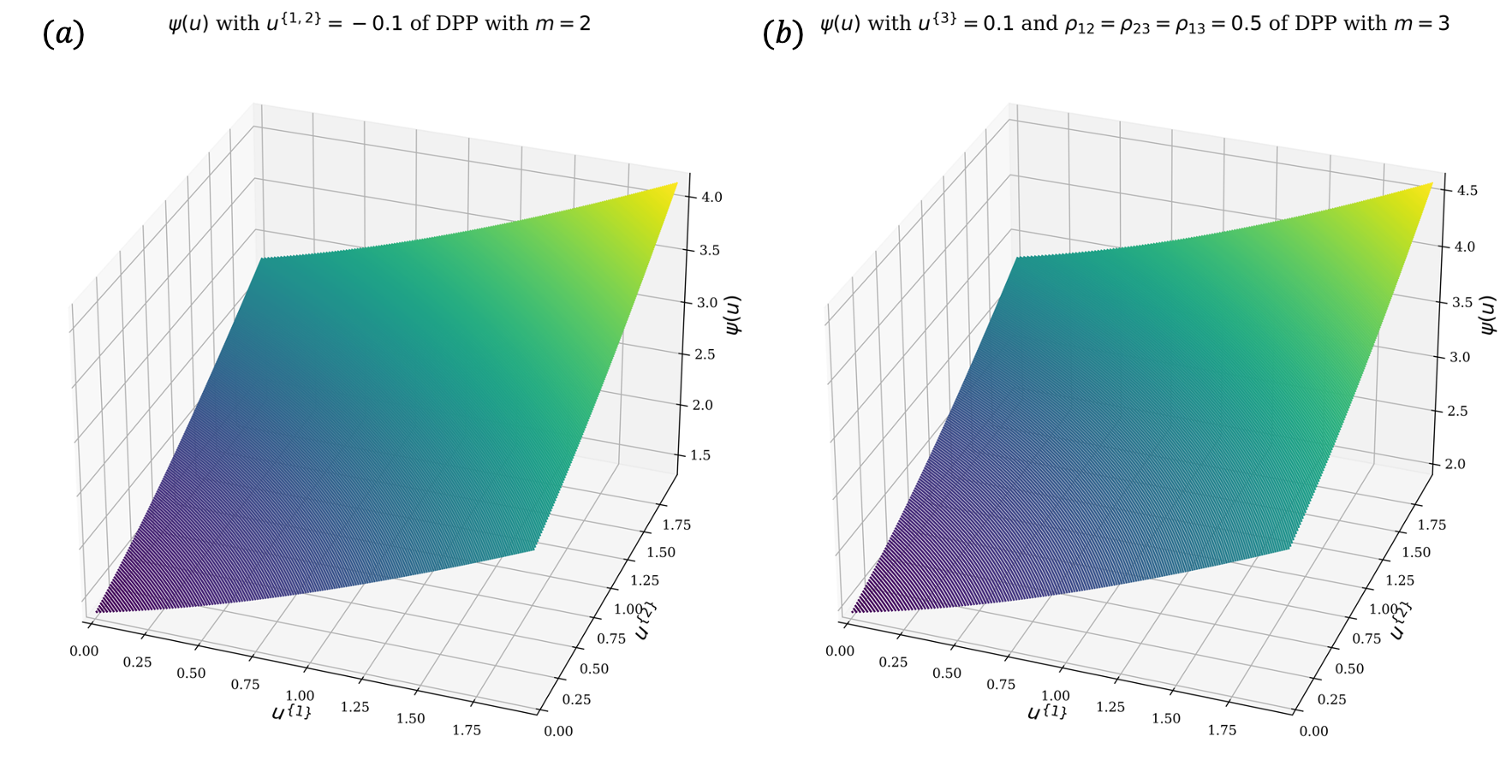}
    \caption{Contour plots of conditional potential functions $\psi(u)$ for DPPs with $m=2$ and $m=3$. (a) Conditional potential function for the DPP model with $m=2$. The evaluation fixes the value of $u^{\{1,2\}}$ to $-0.1$; (b)
    Conditional potential function for the DPP model with $m=3$. The evaluation fixes the values of $u^{\{3\}},u^{\{1,2\}},u^{\{2,3\}},u^{\{3,1\}}$ to $-0.1, \log(1-{0.5}^2)$, respectively.
    }
    \label{fig: conditional potential}
\end{figure}

We check the duality in a marginal kernel $K$ and an $L$-ensemble kernel $L$.
Working with the identity $K=L(L+I_{2\times 2})^{-1}=R(R+D^{-2})^{-1}$,
we obtain the identities
\begin{align*}
K_{11}
&=\frac{e^{-u^{\{2\}}}+e^{u^{\{1,2\}}}}
{e^{-u^{\{1\}}}+e^{-u^{\{2\}}}
+e^{-u^{\{1\}}-u^{\{2\}}}
+e^{u^{\{1,2\}}}},
\quad 
K_{22}
=\frac{e^{-u^{\{1\}}}+e^{u^{\{1,2\}}}}
{e^{-u^{\{1\}}}+e^{-u^{\{2\}}}
+e^{-u^{\{1\}}-u^{\{2\}}}
+e^{u^{\{1,2\}}}}, \\
\mathrm{det}(K)
&=\frac{e^{u^{\{1,2\}}}}
{e^{-u^{\{1\}}}+e^{-u^{\{2\}}}
+e^{-u^{\{1\}}-u^{\{2\}}}
+e^{u^{\{1,2\}}}}.
\end{align*}
Together with the differentiation of the (conditional) potential function $\psi(u)$, we confirm the duality formulae:
\begin{align*}
K_{11}
(=:\eta_{\{1\}})=\frac{\partial \psi(u)}{\partial u^{\{1\}}},\quad
K_{22}
(=:\eta_{\{2\}})
=\frac{\partial \psi(u)}{\partial u^{\{2\}}},\quad \text{and}\quad
\mathrm{det}(K)
(=:\eta_{\{1,2\}})
=\frac{\partial \psi(u)}{\partial u^{\{1,2\}}}.
\end{align*}
Consider the Fisher information matrix  for the mixed parameterization
$\omega= ( \eta_{\{1\}} , \eta_{\{2\}} \,;\, \theta^{\{1,2\}}  )$.
Observe that the DPP also has the form
\[
P_{\theta}(A)=(\eta_{\{1\}}-\eta_{\{1,2\}})1_{A=\{1\}}
+(\eta_{\{2\}}-\eta_{\{1,2\}})1_{A=\{2\}}
+\eta_{\{1,2\}}1_{A=\{1,2\}}
+(1-\eta_{\{1\}}-\eta_{\{2\}}+\eta_{\{1,2\}})1_{A=\emptyset}.
\]
This, together with 
$(\partial /\partial \theta^{\{1,2\}})\log P_{\theta}(A)=T_{\{1,2\}}(A)-\mathrm{E}_{\theta}[T_{\{1,2\}}(A)]$,
gives
\begin{align*}
\mathcal{G}_{\{1\}\{1,2\}}(\omega)
=
\sum_{A\subseteq 2^{\{1,2\}}}P_{\theta}(A) 
\frac{\partial}{\partial \eta_{\{1\}}}\log P_{\theta}(A)
\frac{\partial}{\partial \theta^{\{1,2\}}}\log P_{\theta}(A)
=\frac{\partial \psi(u)}{\partial \eta_{\{1\}}\partial u^{\{1,2\}} }.
\end{align*}

\subsection{Determinantal point process with $m=3$}

We then work with DPP of $m=3$.
In this case, we get the following form of the DPP model as the curved exponential family (\ref{eq:cExpDPP}):
\begin{align*}
P_{u}(A)=e^{\sum_{a=1}^{3}u^{\{a\}}T_{\{a\}}(A)
+\sum_{1\le a<b\le 3}u^{\{a,b\}}T_{\{a,b\}}(A)
+\theta^{\{1,2,3\}}T_{\{1,2,3\}}(A)
-\psi(u)
},
\end{align*}
where
\begin{align*}
T_{\{a\}}(A)&=1_{a\in A},\quad
u^{\{a\}}=\log L_{aa}
\quad (a=1,2,3)\\
T_{\{a,b\}}(A)&=1_{a\in , b\in A}, \quad
u^{\{a,b\}}=\log(1-\rho_{ab}^{2})\quad
(1\le a< b\le 3),
\end{align*}
and
\begin{align*}
\begin{split}
\theta^{\{1,2,3\}}&=
\log\left(e^{ u^{\{1,2\}}}
+
e^{u^{\{2,3\}}}
+
e^{u^{\{1,3\}}}
+2 \sqrt{
(1-e^{u^{\{1,2\}}})
(1-e^{u^{\{2,3\}}})
(1-e^{u^{\{1,3\}}})
}
-2
\right) \\
&\quad -u^{\{1,2\}}
       -u^{\{2,3\}}
       -u^{\{1,3\}}.
       \end{split}
\end{align*}
The conditional potential function $\psi(u)$ is given by
\begin{align*}
\psi(u)&=u^{ \{1\}} + u^{ \{2\}} + u^{ \{3\}}\\
&\quad +\log
\left(
(1+e^{-u^{\{1\}}})
(1+e^{-u^{\{2\}}})
(1+e^{-u^{\{3\}}})
+2\sqrt{ (1-e^{u^{\{1,2\}} })
(1-e^{u^{\{1,3\}} })
(1-e^{u^{\{2,3\}} })
}
\right.
\\
&
\left.
\quad\quad\quad\quad
-(1-e^{u^{\{1,2\}}})(1+e^{ -u^{\{3\}}})
-(1-e^{u^{\{1,3\}}})(1+e^{ -u^{\{2\}}})
-(1-e^{u^{\{2,3\}}})(1+e^{ -u^{\{1\}}})
\right).
\end{align*}

\begin{figure}[htbp]
    \centering
\includegraphics[scale=0.5]{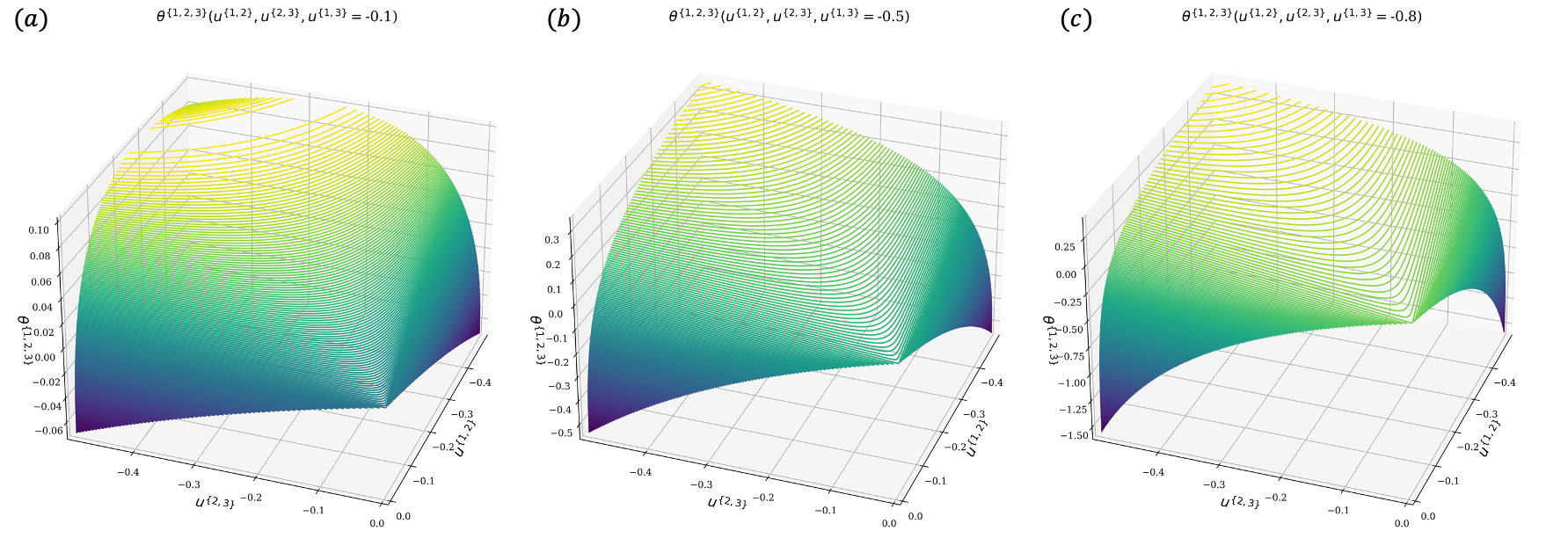}
    \caption{Plots of $\theta^{\{1,2,3\}}(u^{\{1,2\}}, u^{\{2,3\}}, u^{\{1,3\}})$. (a) the plot at $u^{\{1,3\}}=-0.1$; (b) the plot at $u^{\{1,3\}}=-0.5$; (c) the plot at $u^{\{1,3\}}=-0.8$
    }
    \label{fig: curvature of DPP 3}
\end{figure}

Here we see shapes of the conditional potential function $\psi(u)$ and the embedding $\theta^{\{1,2,3\}}(u)$.
Figure \ref{fig: conditional potential} (b) exhibits the shape of the conditional potential function $\psi(u)$.
Although it has a relatively complex shape, the contour plot reveals the convex shape of $\psi(u)$ with respect to $u^{\{1\}}$ and $u^{\{2\}}$.
Figure \ref{fig: curvature of DPP 3} shows the embedding structure of the DPP model via $\theta^{\{1,2,3\}}(u^{\{1,2\}}, u^{\{2,3\}}, u^{\{1,3\}})$.
The embedding $\theta^{\{1,2,3\}}(u^{ \{1,2\} }, u^{\{2,3\}}, u^{\{1,3\}})$ is symmetric as 
described by the explicit form
and depicted by Figure \ref{fig: curvature of DPP 3}.
The gradient of $\theta^{\{1,2,3\}}(u^{ \{1,2\} }, u^{\{2,3\}}, u^{\{1,3\}})$ with respect to $(u^{ \{1,2\} }, u^{\{2,3\}})$ heavily depends on $u^{\{1,3\}}$ as in Figure \ref{fig: curvature of DPP 3}.

\section{Proofs}
\label{sec: proofs}

This section presents the proofs for Theorem \ref{thm:DPPiscurvedexp}, 
and 
Corollaries \ref{cor: e-curvature} and \ref{cor: Fisher w.r.t. u}.

\begin{proof}[Proof of Theorem \ref{thm:DPPiscurvedexp}]

Before the main proof, 
note that 
by using the diagonal scaling $L=DRD$,
$\log \mathrm{det}(L+I_{m\times m})$ is written as $\psi(u)$ in the following manner:
\begin{align*}
\log \mathrm{det}(L+I_{m\times m})
&=\log \mathrm{det}D^{2}+\log\mathrm{det}(R+D^{-2})\\
&=\sum_{i=1}^{m}\log L_{ii}
+ \log\mathrm{det}(R+D^{-2})\\
&=\psi(u).
\end{align*}

For the main proof,
we consider $P_{u}(A)$ in four cases:
$|A|=1$, $|A|=2$, $|A|= 3$, $|A|\ge 4$.

\textbf{Case 1 ($A=\{i\}$, $1\le i \le m$)}:
For $A=\{i\}$, the probability $P_{u}(A)$ is written as
\begin{align*}
P_{u}(A)&=\exp(\log L_{ii} - \psi(u))
=\exp(u^{\{i\}} - \psi(u))
=\exp\left(\sum_{\alpha_{1}\in\mathcal{S}_{1}^{m}}
u^{\alpha_{1}}
T_{\alpha_{1}}(A) -\psi(u)\right),
\end{align*}
which completes Case 1.

\textbf{Case 2 ($A=\{i,j\}$, $1\le i <j \le m$)}:
Consider $A=\{i,j\}$.
By the diagonal scaling of $L$,
the probability $P_{u}(A)$ is written as
\begin{align*}
P_{u}(A)&=
\exp(\log \mathrm{det}L_{\{i,j\}} - \psi(u))
=\exp(\log \mathrm{det}D_{\{i,j\}}R_{\{i,j\}}
D_{\{i,j\}}
- \psi(u)).
\end{align*}
This, together with $\mathrm{det}R_{\{i,j\}}=1-\rho_{ij}^{2}$, yields
\begin{align*}
P_{u}(A)&=\exp(u^{\{i\}}+u^{\{j\}}+
u^{\{i,j\}}-\psi(u) )
=\exp\left(
\sum_{\alpha_{1}\in\mathcal{S}_{1}^{m}}
u^{\alpha_{1}}
T_{\alpha_{1}}(A)
+\sum_{\alpha_{2}\in\mathcal{S}_{2}^{m}}
u^{\alpha_{2}}T_{\alpha_{2}}(A)-\psi(u)
\right),
\end{align*}
which completes Case 2.

\textbf{Case 3 ($A=\{i,j,k\}$, $1\le i <j<k \le m$)}:
Consider $A=\{i,j\}$.
By the same procedure as before, we get
\begin{align*}
P_{u}(A)=\exp(\log \mathrm{det}L_{\{i,j,k\}} -\psi(u))
=\exp\left(
u^{ \{i\} }+u^{ \{j\} }+u^{ \{k\} }
+\log \mathrm{det}R_{\{i,j,k\}} -\psi(u)
\right).
\end{align*}
Observe $A=\{i,j\}\cup\{j,k\}\cup\{k,i\}$.
Then 
\[
u^{  \{i,j\} }T_{ \{i,j\} }(A)
+
u^{ \{j,k\} }T_{ \{j,k\} }(A)
+u^{ \{i,k\} }T_{ \{i,k\} }(A)
=\sum_{\alpha_{2}\subset \{i,j,k\} }
u^{\alpha_{2}}T_{\alpha_{2}}(A)
=u^{\{i,j\}}+u^{\{j,k\}}+u^{ \{i,k\}} 
\]
and thus we get
\begin{align*}
P_{u}(A)=\exp\left(
\sum_{\alpha_{1}\in\mathcal{S}_{1}^{m}}u^{\alpha_{1}}T_{\alpha_{1}}(A)
+\sum_{\alpha_{2}\in\mathcal{S}_{2}^{m} }u^{\alpha_{2}}T_{\alpha_{2}}(A)
+\log \mathrm{det}R_{\{i,j,k\}}-
\sum_{\alpha_{2}\subset (i,j,k)}u^{\alpha_{2}}T_{\alpha_{2}}(A)
\right).
\end{align*}
Together with the formula
\[
\log \mathrm{det} R_{\{i,j,k\}}
=\log \mathrm{det} 
\left(
I_{3\times 3}+
(
\sqrt{1-e^{u^{\{a,b\} }}}
)_{\{a,b\}\subset \{i,j,k\}}
\right),
\]
we complete Case 3.

Case 4 ($|A|=k \ge 4$): 
For $A=\mathcal{I}_{k}$ and $l\le k$, 
we decompose $A$ as 
\[
A=\mathcal{I}_{k}=\cup_{\mathcal{J}_{l}\subset \mathcal{I}_{k}}\mathcal{J}_{l}
\]
and we get
\[
\sum_{\mathcal{J}_{l}\subset \mathcal{I}_{k}}
\theta^{\mathcal{J}_{l}}(u)T_{\mathcal{J}_{l}}(A)
=
\sum_{\mathcal{J}_{l}\in\mathcal{S}_{l}^{m}}
\theta^{\mathcal{J}_{l}}(u)T_{\mathcal{J}_{l}}(A)
=\sum_{\mathcal{J}_{l}\subset \mathcal{I}_{k}}
\theta^{\mathcal{J}_{l}},
\]
which, together with the same procedure as in Case 3, completes the proof of (a).
The proof of (b) is obvious.
\end{proof}

\begin{proof}[Proof of Corollary \ref{cor: e-curvature}]

Observe that we have
\begin{align*}
B=\frac{\partial \theta }{\partial u}=
\begin{pmatrix}
    I_{m\times m} & 0_{m \times m(m-1)/2} \\
    0_{m(m-1)/2 \times m} & I_{m(m-1)/2 \times m(m-1)/2}\\
    0_{ (2^{m}-1-m(m+1)/2) \times m} & C
\end{pmatrix}
\end{align*}
where $C$ is the $((2^{m}-1-m(m+1)/2) \times m(m-1)/2)$-matrix that corresponds to $\partial \theta^{\mathcal{I}_{k}} /\partial u^{\alpha_{2}}$ ($\mathcal{I}_{k}\in\mathcal{S}_{k}^{m}$, $3\le k \le m$, $\alpha_{2}\in\mathcal{S}_{2}^{k}$).
So, the Jacobian matrix $\partial \theta/\partial u$ does not depend on $(u^{\alpha_{1}})_{\alpha_{1}\in\mathcal{S}_{1}^{m}}$.
Together with equation (\ref{eq: e-curvature}), this completes the proof.
\end{proof}

\begin{proof}[Proof of Corollary \ref{cor: Fisher w.r.t. u}]

Consider (a).
Using 
the expression of the score vector of the DPP model with respect to $(u^{\alpha_{1}})_{\alpha_{1}\in\mathcal{S}_{1}^{m}}$
\[
(\partial / \partial u^{\alpha_{1}}) \log P_{u}(A)
= T_{\alpha_{1}}(A)- (\partial/ \partial u^{\alpha_{1}}) \psi(u)
\quad (\alpha_{1}\in\mathcal{S}_{1}^{m}),
\]
and the expression of the Fisher information matrix
$\mathcal{G}(u)= \mathrm{E}_{\theta(u)}[ - \nabla_{u}^{2} \log P_{u}(A)]$,
we get 
\[
\mathcal{G}_{\alpha_{1} \tilde{\alpha}_{1}}(u)
=
\frac{\partial^{2} \psi(u)}{\partial u^{\alpha_{1}}\partial u^{\tilde{\alpha}_{1}}}.
\]
Next, observe that the Jacobian matrix has the form
\begin{align}
B=\frac{\partial \theta }{\partial u}=
\begin{pmatrix}
    I_{m\times m} & 0_{m \times m(m-1)/2} \\
    0_{m(m-1)/2 \times m} & I_{m(m-1)/2 \times m(m-1)/2}\\
    0_{ (2^{m}-1-m(m+1)/2) \times m} & C
\end{pmatrix}
\label{eq: Jacobian}
\end{align}
where $C$ is the $((2^{m}-1-m(m+1)/2) \times m(m-1)/2)$-matrix that corresponds to $\partial \theta^{\mathcal{I}_{k}} /\partial u^{\alpha_{2}}$ ($\mathcal{I}_{k}\in\mathcal{S}_{k}^{m}$, $3\le k \le m$, $\alpha_{2}\in\mathcal{S}_{2}^{k}$).
Denote the Fisher information matrix with respect to $\theta$ of the log-linear model evaluated at $\theta=\theta(u)$ by $M$ and write it in the block-matrix notion:
\begin{align*}
    M=\begin{pmatrix}
        M_{ss} & M_{sp} & M_{sh}\\
        M_{ps} & M_{pp} & M_{ph}\\
        M_{hs} & M_{hp} & M_{hh}
    \end{pmatrix},
\end{align*}
where $M_{ss}$ is an $m\times m$ sub-matrix, $M_{sp}$ is an $m\times m(m-1)/2$ sub-matrix, $M_{sh}$ is an $m\times (2^{m}-1-m(m+1)/2)$ sub-matrix,
and the remaining sub-matrix are defined in similar ways.
Then the expression (\ref{eq: Fisher w.r.t.u}) gives
\begin{align*}
\mathcal{G}(u)=B^{\top}MB
=\begin{pmatrix}
    M_{ss} & M_{sp}+M_{sh}C \\
    M_{ps}+C^{\top}M_{hs}  &  *
\end{pmatrix},
\end{align*}
where $*$ is a suitable matrix.
By using equation (\ref{eq: Fisher information of exp}),
for $\alpha_{1},\tilde{\alpha}_{1}\in\mathcal{S}_{1}^{m}$,
the
$(\alpha_{1},\tilde{\alpha}_{1})$-component of $M_{ss}$ is expressed as
\begin{align*}
(M_{ss})_{\alpha_{1} \tilde{\alpha}_{1}}
=\mathrm{Cov}_{\theta(u)}[T_{\alpha_{1}}
,T_{\tilde{\alpha}_{1}}
]
=\mathrm{det}(K_{\alpha_{1} \cup\tilde{\alpha}_{1}})-\mathrm{det}(K_{\alpha_{1}})
\mathrm{det}(K_{\tilde{\alpha}_{1}}),
\end{align*}
which concludes (a).
Noting that $\partial \theta^{\mathcal{I}_{k}} / \partial u^{\alpha_{2}}$ ($\mathcal{I}_{k}\in\mathcal{S}_{k}^{m}$, $3\le k\le m$, $\alpha_{2}\in\mathcal{S}_{2}^{m}$) is zero if $\alpha_{2}\subset \mathcal{I}_{k}$ (see also Figure \ref{fig: curvedexponential}),
we obtain (b).
\end{proof}

\section{Acknowledgement}
The authors thank Takahiro Kawashima for helpful comments to the early version of this work.
This work is supported by JSPS KAKENHI (21H05205, 21K12067, 23K11024, JP22H03653, 23H04483), MEXT (JPJ010217).

\bibliographystyle{plainnat}
\bibliography{references}

\end{document}